\documentclass[12pt,a4paper,reqno]{amsart}
\allowdisplaybreaks

\usepackage[utf8]{inputenc}
\usepackage{amsmath,amssymb,amsthm,amsfonts,latexsym}
\usepackage{mathrsfs}
\usepackage{enumerate}
\usepackage{url}
\usepackage{cases}
\usepackage{hyperref}
\usepackage{etoolbox}

\hypersetup{colorlinks=true,citecolor=blue,linkcolor=blue,urlcolor=blue}
\numberwithin{equation}{section}

\newcommand{\N}{\mathbb N}
\newcommand{\Z}{\mathbb Z}
\newcommand{\R}{\mathbb R}

\theoremstyle{plain}
\newtheorem{theorem}{Theorem}
\newtheorem{lemma}{Lemma}
\newtheorem{proposition}{Proposition}

\newtheorem{conjecture}{Conjecture}
\theoremstyle{remark}

\makeatletter
\patchcmd{\@settitle}{\uppercasenonmath\@title}{}{}{}
\patchcmd{\@setauthors}{\MakeUppercase}{}{}{}
\patchcmd{\section}{\scshape}{}{}{}
\makeatother

\begin{document}

\title[An Improvement of Konstantoulas' Density Constant]
{An Improvement of Konstantoulas' Density Constant}

\author[H. Li and Z. Zhang]{Huixi Li and Zihan Zhang}

\address{(Huixi Li) School of Mathematical Sciences and LPMC, Nankai University, Tianjin
300071, People's Republic of China}
\email{lihuixi@nankai.edu.cn}

\address{(Zihan Zhang) School of Mathematical Sciences and LPMC, Nankai University, Tianjin
300071, People's Republic of China}
\email{2211056@mail.nankai.edu.cn}

\keywords{Erd\H{o}s--Tur\'an conjecture; additive bases; representation functions; generating functions}

\subjclass[2020]{11B13, 11B34}

\begin{abstract}
Let \(A\subseteq \N\), and define its ordered representation function
\[
r(n)=\#\{(a,b)\in A\times A:a+b=n\}.
\] 
The Erd\H{o}s--Tur\'an conjecture asserts that, if \(r(n)>0\) for all
sufficiently large \(n\), then \(r(n)\) is unbounded. Konstantoulas proved a
density-theoretic version: if the upper density of
\(E=\N\setminus(A+A)\) is less than \(1/10\), then
\(\limsup_{n\to\infty} r(n)>5\). 
In this paper, we improve Konstantoulas' constant to \(7/32\). 
We also prove that \(D(E)<1/2\) implies
\(\limsup_{n\to\infty} r(n)>3\), 
and give a conditional criterion forcing
\(\limsup_{n\to\infty} r(n)>7\).
\end{abstract}

\maketitle

\section{Introduction}

Throughout the paper, \(\N\) denotes the set of positive integers. A set
\(A\subseteq\N\) is called an asymptotic basis of order \(2\) if every
sufficiently large integer can be written as a sum of two elements of \(A\).
We write
\[
r(n)=\#\{(a,b)\in A\times A:a+b=n\}
\]
for the ordered representation function. The Erd\H{o}s--Tur\'an conjecture,
first proposed in 1941~\cite{ErTu}, asserts the following.

\begin{conjecture}[Erd\H{o}s-Tur\'an]\label{conj:ET}
If \(A\) is an asymptotic basis of order \(2\), then \(r(n)\) is unbounded.
Equivalently,
\[
\limsup_{n\to\infty} r(n)=\infty .
\]
\end{conjecture}

Even very small lower bounds for \(\limsup_{n\to\infty} r(n)\) are difficult. 
Grekos, Haddad, Helou and Pihko~\cite{Grekos} in 2003 proved that,
for bases of order \(2\) in the stronger sense that every positive integer
is represented, \(r(n)\ge 6\) for infinitely many \(n\).
Borwein, Choi and Chu~\cite{Borwein} in 2006 improved this lower bound from
\(6\) to \(8\), again under this stronger basis assumption.
These two results use the full-basis assumption in an essential
way and are computational in nature. Konstantoulas~\cite{Kons} considered a
different, density-theoretic setting: if the exceptional set $E=\N\setminus(A+A)$ 
has upper density less than \(1/10\), then \(\limsup_{n \to \infty} r(n)>5\). 

Several related variants have also been studied. 
In the setting of \(\N\), Chen~\cite{Chen2012} constructed an asymptotic basis of order \(2\) whose two-fold representation function has the minimal value on a density-one set. 
Over \(\Z\), Nathanson~\cite{Na} constructed arbitrarily sparse unique representation bases, showing that the direct integer analogue of the Erd\H{o}s--Tur\'an conjecture is false. 
Finite cyclic analogues over \(\Z/m\Z\) also exhibit bounded representation phenomena: Ruzsa~\cite{Ruzsa} introduced the underlying finite construction, Chen~\cite{Chen2008} proved the uniform bound \(R_m\le 288\) for Ruzsa's numbers, and later work of S\'andor--Yang~\cite{Sandor} and Ding--Zhao~\cite{Ding} refined the corresponding lower and upper bounds.

The present paper remains in the setting \(A\subseteq\N\). For \(E\subseteq\N\), write
\[
D(E)=\limsup_{n\to\infty}
\frac{|E\cap\{1,2,\ldots,n\}|}{n}
\]
for its upper density. Our first result improves Konstantoulas' constant.

\begin{theorem}\label{thm:main}
Let \(A\subseteq\N\), and let \(E=\N\setminus(A+A)\). If
\[
D(E)<\frac{7}{32},
\]
then
\[
\limsup_{n\to\infty} r(n)>5.
\]
\end{theorem}

The improvement from \(1/10\) to \(7/32\) comes from one simple structural
observation. If \(r(n)\) is eventually bounded, then the levels on which
\(r(n)\) is odd have zero Abel density. This leaves only the even levels in the main term of the final \(L^1/L^2\) estimate. Optimizing the remaining
one-parameter inequality gives the constant \(7/32\).

The same argument gives the following lower-threshold analogue.
\begin{theorem}\label{thm:three}
Let \(A\subseteq\N\), and let \(E=\N\setminus(A+A)\). If
\[
D(E)<\frac12,
\]
then
\[
\limsup_{n\to\infty} r(n)>3.
\]
\end{theorem}

Theorem~\ref{thm:three} gives a quick, but weak, consequence for Sidon sets.
Here a Sidon set means that all unordered sums \(a+b\), \(a\le b\), are
distinct; equivalently, for the ordered representation function one has
\(r(n)\le 2\) for all \(n\). Thus Theorem~\ref{thm:three} 
excludes the possibility that
\(D(E)<1/2\) for an infinite Sidon set. This is much weaker than the
classical theorem of Erd\H{o}s~\cite{HalberstamRoth}: for some absolute
constant \(c\),
\[
\liminf_{n\to\infty}
\frac{\lvert A\cap \{1,2,\ldots,n\}\rvert}{n^{1/2}}(\log n)^{1/2}\leq c,
\]
which implies that an infinite Sidon set has
a subsequence along which \(|(A+A)\cap[1,n]|/n\to0\). 

The next result records the additional information on the level \(r(n)=4\) that is needed when the same method is pushed toward the stronger conclusion \(\limsup_{n \to \infty} r(n)>7\).
\begin{theorem}\label{thm:seven}
Let \(A\subseteq\N\), let \(E=\N\setminus(A+A)\), and put
\[
N_4=N_4(A)=\{n\in\N:r(n)=4\},\qquad
S_4(z)=\sum_{n\in N_4}z^n.
\]
Define the lower Abel density
\[
\underline\delta_4=\liminf_{t\to1^-}(1-t)S_4(t).
\]
If
\[
\underline\delta_4>\frac{1}{16}+3D(E),
\]
then
\[
\limsup_{n\to\infty} r(n)>7.
\]
\end{theorem}
The paper is organized as follows. In Section~\ref{sec:setup} we record the common generating-function setup
and prove the vanishing of the odd levels. In Section~\ref{sec:proofs} we
prove Theorems~\ref{thm:main} and~\ref{thm:three}, and explain why the constants in the statement of Theorems~\ref{thm:main} and~\ref{thm:three} cannot be improved within this one-parameter method.
Finally, in Section~\ref{sec:thresholds} we discuss higher representation
thresholds and prove Theorem~\ref{thm:seven}.

\section{The Generating-Function Setup}\label{sec:setup}

Let \(A\subseteq\N\). Suppose, for contradiction, that
\[
r(n)\le k
\]
for all sufficiently large \(n\). For \(0\le j\le k\), let
\[
N_j=N_j(A)=\{n\in\N:r(n)=j\},
\qquad
S_j(z)=\sum_{n\in N_j} z^n.
\]
Then \(N_0=E\). Let
\[
g(z)=\sum_{a\in A} z^a
\]
be the generating function of \(A\). After absorbing finitely many initial
terms into polynomials, we have
\begin{align}
S_0(z)+\sum_{i=1}^k S_i(z)
  &=\frac{1}{1-z}-P_{1,k}(z), \label{eq:partition}\\
g(z)^2
  &=P_{2,k}(z)+\sum_{i=1}^k iS_i(z), \label{eq:representation}
\end{align}
where \(P_{1,k}\) and \(P_{2,k}\) are polynomials.

We shall use the following elementary Abel-density bound.
\begin{lemma}\label{lem:abel-density}
Let \(E\subseteq\N\), let \(S_E(t)=\sum_{n\in E}t^n\), and let \(D(E)\) be
the upper density of \(E\). Then
\[
\limsup_{t\to1^-}(1-t)S_E(t)\le D(E).
\]
\end{lemma}

\begin{proof}
See \cite[Lemma 2 and Corollary 3]{Kons}. Let \(E(x)=|E\cap\{1,\ldots,x\}|\). Fix \(\eta>0\). For all sufficiently
large \(n\), we have \(E(n)\le (D(E)+\eta)n\). Since
\[
S_E(t)=(1-t)\sum_{n\ge1}E(n)t^n,
\]
it follows, after separating finitely many initial terms, that
\[
\limsup_{t\to1^-}(1-t)S_E(t)\le D(E)+\eta.
\]
Letting \(\eta\to0\) proves the lemma.
\end{proof}

Fix \(\varepsilon>0\). By Lemma~\ref{lem:abel-density}, there is a sequence
\(r_m\to1^-\) such that
\[
S_0(r_m^2)\le \frac{D(E)+\varepsilon}{1-r_m^2}.
\]
Passing to a subsequence if necessary, we may assume that all limits
\[
\ell_j=\lim_{m\to\infty} (1-r_m^2)S_j(r_m^2),
\qquad 0\le j\le k,
\]
exist. Then \eqref{eq:partition} gives
\begin{equation}\label{eq:ellsum}
\sum_{j=0}^k \ell_j=1,
\qquad
\ell_0\le D(E)+\varepsilon.
\end{equation}

The following proposition is the reason the constant in Konstantoulas' theorem can be improved here: once the odd representation levels have zero limiting mass, the final optimization is carried out only over the even levels.
\begin{proposition}\label{prop:odd}
For each odd \(j\), \(1\le j\le k\), one has
\[
\ell_j=0.
\]
\end{proposition}

\begin{proof}
For each \(n\), the non-diagonal representations \(n=a+b\), \(a\ne b\), occur in pairs \((a,b)\) and \((b,a)\). Hence the parity of \(r(n)\) is determined by the diagonal representation \(n=2a\), if such a representation exists. Apart from the finitely many \(n\) for which \(r(n)>k\), this gives
\begin{equation}\label{eq:oddidentity}
\sum_{\substack{1\le i\le k\\ i\ {\rm odd}}}S_i(z)=g(z^2)-P_{3,k}(z),
\end{equation}
where \(P_{3,k}\) is a polynomial accounting for the exceptional initial terms.

For \(0<r<1\), \eqref{eq:oddidentity} gives
\[
\sum_{\substack{1\le i\le k\\ i\ {\rm odd}}}S_i(r^2)
\le g(r^4)+O(1).
\]
On the other hand, the assumption \(r(n)\le k\) eventually implies
\[
g(r)^2=\sum_{n\ge0}r(n)r^n=O\left(\frac{1}{1-r}\right)
\qquad (r\to1^-).
\]
Hence \(g(r^4)=O((1-r^4)^{-1/2})\). Therefore
\[
(1-r^2)\sum_{\substack{1\le i\le k\\ i\ {\rm odd}}}S_i(r^2)\to0
\qquad (r\to1^-),
\]
and so every odd \(\ell_j\) is zero.
\end{proof}

We shall also use the following one-parameter form of Konstantoulas'
\(L^1/L^2\) estimate.

\begin{lemma}\label{lem:l1l2}
Assume that \(r(n)\le 2h+1\) for all sufficiently large \(n\). Let
\(\lambda\in\R\), and suppose that \(\sum_{s=1}^h 2s\,\ell_{2s}>0\). With
the notation above, we have 
\begin{equation}\label{eq:l1l2-general}
\sqrt{
\frac{\displaystyle
\sum_{s=0}^h (2s-\lambda)^2\ell_{2s}}
{\displaystyle
\sum_{s=1}^h 2s\,\ell_{2s}}
} \ge 1.
\end{equation}
\end{lemma}

\begin{proof}
Put \(K=2h+1\). Combining \eqref{eq:partition} and \eqref{eq:representation},
we get
\begin{align}
g(z)^2
&=P_{2,K}(z)+\lambda\bigl(S_0(z)+\cdots+S_K(z)\bigr)
  +\sum_{j=0}^K(j-\lambda)S_j(z) \notag\\
&=\frac{\lambda}{1-z}+P_{2,K}(z)-\lambda P_{1,K}(z)
  +\sum_{j=0}^K(j-\lambda)S_j(z). \label{eq:lambda-decomp}
\end{align}

Let \(z=re^{i\theta}\). Since \(|g(z)^2|=|g(z)|^2\), we have 
\[
\int_{-\pi}^{\pi}|g(re^{i\theta})^2|\,d\theta
=\int_{-\pi}^{\pi}|g(re^{i\theta})|^2\,d\theta
=2\pi g(r^2).
\]
The polynomial terms in \eqref{eq:lambda-decomp} contribute \(O(1)\) to this
integral, and
\[
\int_{-\pi}^{\pi}\left|\frac{1}{1-re^{i\theta}}\right|\,d\theta
=O\left(\log\frac{1}{1-r}\right).
\]
By Proposition~\ref{prop:odd} and Cauchy's inequality,
\[
\int_{-\pi}^{\pi}|S_j(re^{i\theta})|\,d\theta
\le 2\pi S_j(r^2)^{1/2}
=o\bigl((1-r^2)^{-1/2}\bigr)
\]
for every odd \(j\), along the chosen sequence \(r=r_m\).

Thus the main contribution comes from the even levels. Cauchy's inequality
and the disjointness of the coefficient supports of \(S_0,S_2,\ldots,S_{2h}\)
give
\begin{align*}
\int_{-\pi}^{\pi}
\left|
\sum_{s=0}^h(2s-\lambda)S_{2s}(re^{i\theta})
\right|\,d\theta  
\qquad\le
2\pi
\left(
\sum_{s=0}^h(2s-\lambda)^2S_{2s}(r^2)
\right)^{1/2}.
\end{align*}
Consequently,
\[
g(r^2)
\le
\left(
\sum_{s=0}^h(2s-\lambda)^2S_{2s}(r^2)
\right)^{1/2}
+o\bigl((1-r^2)^{-1/2}\bigr)
\]
along the chosen sequence \(r=r_m\).

Finally, \eqref{eq:representation} and Proposition~\ref{prop:odd} imply
\[
(1-r^2)g(r^2)^2
\to
\sum_{s=1}^h 2s\,\ell_{2s}.
\]
Multiplying the preceding inequality by \((1-r^2)^{1/2}\) and letting
\(r=r_m\to1^-\), we obtain
\[
\left(\sum_{s=1}^h 2s\,\ell_{2s}\right)^{1/2}
\le
\left(\sum_{s=0}^h (2s-\lambda)^2\ell_{2s}\right)^{1/2}.
\]
Since the denominator is positive by hypothesis, this is
\eqref{eq:l1l2-general}.
\end{proof}

\section{Proofs of Theorems~\ref{thm:main} and~\ref{thm:three}}\label{sec:proofs}

\begin{proof}[Proof of Theorem~\ref{thm:main}]
Assume, to the contrary, that \(r(n)\le5\) for all sufficiently large \(n\).
Choose \(\varepsilon>0\) so small that
\[
D(E)+\varepsilon<\frac{7}{32}.
\]
Form the sequence \(r_m\to1^-\) and the limits \(\ell_j\) as in
Section~\ref{sec:setup}. By Proposition~\ref{prop:odd},
\[
\ell_1=\ell_3=\ell_5=0.
\]
Hence, by \eqref{eq:ellsum},
\[
\ell_0+\ell_2+\ell_4=1,
\qquad
\ell_0\le D(E)+\varepsilon<\frac{7}{32}.
\]
In particular,
\[
\ell_2+\ell_4=1-\ell_0>0,
\]
so the denominator \(2\ell_2+4\ell_4\) in Lemma~\ref{lem:l1l2} is positive.
Taking \(h=2\) and \(\lambda=5/2\) in Lemma~\ref{lem:l1l2}, we obtain
\begin{equation}\label{eq:main-lambda}
1\le
\sqrt{
\frac{
\frac{25}{4}\ell_0+\frac14\ell_2+\frac94\ell_4
}{
2\ell_2+4\ell_4
}
}.
\end{equation}
On the other hand,
\[
\frac{25}{4}\ell_0+\frac14\ell_2+\frac94\ell_4
<
2\ell_2+4\ell_4.
\]
Indeed, after multiplying by \(4\), this inequality is equivalent to
\[
25\ell_0+\ell_2+9\ell_4<8\ell_2+16\ell_4,
\]
or
\[
25\ell_0<7(\ell_2+\ell_4)=7(1-\ell_0),
\]
which is precisely \(\ell_0<7/32\). Thus the right hand side of
\eqref{eq:main-lambda} is strictly smaller than \(1\), a contradiction.
\end{proof}

\begin{proof}[Proof of Theorem~\ref{thm:three}]
Assume, to the contrary, that \(r(n)\le3\) for all sufficiently large \(n\).
Choose \(\varepsilon>0\) so small that
\[
D(E)+\varepsilon<\frac12.
\]
Form the sequence \(r_m\to1^-\) and the limits \(\ell_j\) as in
Section~\ref{sec:setup}. By Proposition~\ref{prop:odd},
\[
\ell_1=\ell_3=0.
\]
Hence, by \eqref{eq:ellsum},
\[
\ell_0+\ell_2=1,
\qquad
\ell_0\le D(E)+\varepsilon<\frac12.
\]
In particular, \(\ell_2=1-\ell_0>0\), so Lemma~\ref{lem:l1l2} applies.
Taking \(h=1\) and \(\lambda=1\), we get
\[
1\le
\sqrt{
\frac{\ell_0+\ell_2}{2\ell_2}
}
=
\sqrt{\frac{1}{2(1-\ell_0)}}.
\]
The last expression is strictly smaller than \(1\), a contradiction.
\end{proof}

We now explain why the constant \(7/32\) in Theorem~\ref{thm:main} cannot be improved merely by
changing the coefficient \(\lambda\) in Lemma~\ref{lem:l1l2} in the
one-parameter decomposition above.
After the odd classes have been shown to be negligible, the relevant data for
the proof of Theorem~\ref{thm:main} are
\[
\ell_0,\ell_2,\ell_4\ge0,
\qquad
\ell_0+\ell_2+\ell_4=1.
\]
The method needs a choice of \(\lambda\) for which
\begin{equation}\label{eq:needed}
\lambda^2\ell_0+(2-\lambda)^2\ell_2+(4-\lambda)^2\ell_4
<
2\ell_2+4\ell_4.
\end{equation}
For fixed \(\ell_0,\ell_2,\ell_4\), the left hand side is minimized when
\[
\lambda=2\ell_2+4\ell_4.
\]
The threshold \(7/32\) is sharp for this optimization. Indeed, take
\[
\ell_0=\frac{7}{32},\qquad
\ell_2=\frac{5}{16},\qquad
\ell_4=\frac{15}{32}.
\]
Then
\[
\ell_0+\ell_2+\ell_4=1,
\qquad
2\ell_2+4\ell_4=\frac52.
\]
The minimizing parameter is \(\lambda=5/2\), and
\[
\left(\frac52\right)^2\frac{7}{32}
+\left(2-\frac52\right)^2\frac{5}{16}
+\left(4-\frac52\right)^2\frac{15}{32}
=\frac52.
\]
Therefore equality holds in \eqref{eq:main-lambda}. The argument cannot
produce a contradiction at \(\ell_0=7/32\), and hence cannot prove any
threshold larger than \(7/32\) without additional information.

The same optimization explains the sharpness of the constant \(1/2\) in
Theorem~\ref{thm:three} for this method. In that case the limiting data are
\[
\ell_0,\ell_2\ge0,\qquad \ell_0+\ell_2=1,
\]
and Lemma~\ref{lem:l1l2} with \(h=1\) requires
\[
\min_{\lambda}\bigl(\lambda^2\ell_0+(2-\lambda)^2\ell_2\bigr)\ge 2\ell_2.
\]
The minimum is attained at \(\lambda=2\ell_2\) and equals
\[
4\ell_0\ell_2.
\]
Thus equality occurs when \(\ell_0=\ell_2=1/2\). Consequently this
one-parameter argument cannot prove Theorem~\ref{thm:three} with any density
threshold larger than \(1/2\).

\section{Proof of Theorem~\ref{thm:seven}}\label{sec:thresholds}

We recall the notation from Theorem~\ref{thm:seven}:
\[
N_4=N_4(A)=\{n\in\N:r(n)=4\},
\qquad
S_4(z)=\sum_{n\in N_4}z^n,
\]
and define the lower Abel density
\[
\underline\delta_4
=\liminf_{t\to1^-}(1-t)S_4(t).
\]
The calculation below shows that this method needs extra mass on the level
\(r(n)=4\) in order to force the stronger conclusion \(r(n)>7\).

\begin{proof}[Proof of Theorem~\ref{thm:seven}]
Assume, to the contrary, that \(r(n)\le7\) for all sufficiently large \(n\).
Choose \(\varepsilon>0\) so small that
\[
\underline\delta_4>\frac{1}{16}+3(D(E)+\varepsilon).
\]
As before, after passing to a suitable sequence \(r_m\to1^-\), all limits
\[
\ell_j=\lim_{m\to\infty}(1-r_m^2)S_j(r_m^2)
\]
exist for \(j=0,2,4,6\), and the odd classes are lower order. Hence
\[
\ell_0+\ell_2+\ell_4+\ell_6=1,
\qquad
\ell_0\le D(E)+\varepsilon,
\qquad
\ell_4\ge \underline\delta_4
\]
by the definition of the lower Abel density. In particular, \(\ell_4>0\), so
the denominator condition in Lemma~\ref{lem:l1l2} is satisfied.
Taking \(h=3\) in Lemma~\ref{lem:l1l2}, the optimized parameter argument
gives the necessary inequality
\begin{equation}\label{eq:higher-necessary}
\min_{\lambda}
\left(
\lambda^2\ell_0+(2-\lambda)^2\ell_2+(4-\lambda)^2\ell_4
+(6-\lambda)^2\ell_6
\right)
\ge
2\ell_2+4\ell_4+6\ell_6.
\end{equation}
Write
\[
\mu=2\ell_2+4\ell_4+6\ell_6.
\]
The minimum on the left of \eqref{eq:higher-necessary} is
\[
4\ell_2+16\ell_4+36\ell_6-\mu^2.
\]
Using \(\ell_0+\ell_2+\ell_4+\ell_6=1\), one checks that
\[
4\ell_2+16\ell_4+36\ell_6
=8\mu-12+12\ell_0-4\ell_4.
\]
Thus the inequality opposite to \eqref{eq:higher-necessary}, namely the
desired contradiction, is
\[
\mu^2-7\mu+12-12\ell_0+4\ell_4>0.
\]
The quadratic \(\mu^2-7\mu+12\) has global minimum \(-1/4\), attained at
\(\mu=7/2\). Therefore a sufficient condition for contradiction is
\[
-\frac14-12\ell_0+4\ell_4>0,
\]
that is,
\[
\ell_4>\frac{1}{16}+3\ell_0.
\]
Since \(\ell_0\le D(E)+\varepsilon\) and
\(\ell_4\ge\underline\delta_4\), our choice of \(\varepsilon\) gives
\[
\ell_4>\frac{1}{16}+3\ell_0,
\]
which yields the contradiction.
\end{proof}

In particular, for an asymptotic basis, where \(D(E)=0\), any counterexample
to the conclusion \(r(n)>7\) must satisfy
\[
\underline\delta_4\le\frac{1}{16}.
\]
Thus, if an asymptotic basis satisfied \(r(n)\le7\) eventually, then the
set of integers with exactly four ordered representations would have lower
Abel density at most \(1/16\).

The same calculation also explains why this one-parameter \(L^1/L^2\) method
cannot prove an unconditional theorem forcing \(r(n)>7\), even when
\(\ell_0=0\). For example, take
\[
\ell_2=\frac58,\qquad
\ell_4=0,\qquad
\ell_6=\frac38.
\]
Then
\[
2\ell_2+4\ell_4+6\ell_6=\frac72.
\]
The left hand side of \eqref{eq:higher-necessary} is minimized at
\[
\lambda=2\ell_2+4\ell_4+6\ell_6=\frac72,
\]
and the minimum is
\[
\left(2-\frac72\right)^2\frac58
+\left(6-\frac72\right)^2\frac38
=\frac{15}{4}.
\]
Since \(15/4>7/2\), no contradiction follows.

\section*{Acknowledgments and AI Disclosure}

Huixi Li's research is supported by the National Natural Science Foundation of China (Grant No. 12561001). 
The authors thank Yuchen Ding for helpful comments on Sidon sets.
OpenAI Codex was used as an auxiliary research tool; in particular, it helped identify the parity observation underlying Proposition~\ref{prop:odd}.

\end{document}